\documentclass[pre,aps,preprint,amsmath,amssymb]{revtex4}
\usepackage{graphicx}
\usepackage{epsfig}
\usepackage{epstopdf}
\include{graphics}
\DeclareMathOperator{\sech}{sech}

\usepackage{amsthm}
\usepackage{multirow}

\newtheorem{thm}{Theorem}[section]
\newtheorem{lm}[thm]{Lemma}

\numberwithin{equation}{section}

\renewcommand\Re{\operatorname{Re}}
\renewcommand\Im{\operatorname{Im}}

\def\xx{{\mathrm x}}

\def\R{\mathbb{R}}
\def\T{\mathcal{T}}

\def\E{\mathcal{E}}

\def\E{\mathcal{E}}

\def\E{{\cal E}}

\begin{document}

\title{Analysis of a Crank-Nicolson
finite difference scheme for $(2+1)$D perturbed nonlinear Schr\"odinger equations with saturable nonlinearity}

\author{Anh-Ha Le$^{1,4}$}

\author{Toan T. Huynh$^{2,4}$}

\author{Quan M. Nguyen$^{3,4}$}

\affiliation{$^{1}$Faculty of Mathematics and Computer Science, University of Science, Ho Chi Minh City, Vietnam}

\affiliation{$^{2}$Department of Mathematics, University of Medicine and Pharmacy at Ho Chi Minh City, Ho Chi Minh City, Vietnam}

\affiliation{$^{3}$Department of Mathematics, International University, Ho Chi Minh City, Vietnam}

\affiliation{$^{4}$Vietnam National University, Ho Chi Minh City, Vietnam}

\date{\today}

\begin{abstract}
We analyze a Crank–Nicolson finite difference discretization for the perturbed (2+1)D nonlinear Schr\"odinger equation with saturable nonlinearity and a perturbation of cubic loss. We show the boundedness, the existence and uniqueness of a numerical solution. We establish the error bound to prove the convergence of the numerical solution. Moreover, we find that the convergence rate is at the second order in both time step and spatial mesh size under a mild assumption. The numerical scheme is validated by the extensive simulations of the (2+1)D saturable nonlinear Schr\"odinger model with cubic loss. The simulations for travelling solitons are implemented by using an accelerated imaginary-time evolution scheme and the Crank–Nicolson finite difference method. 
\end{abstract}


\maketitle


\section{Introduction}\label{Intro}
The nonlinear Schr\"odinger (NLS) equations, which is a nonlinear dispersive wave equation, characterize several physical systems, such as those in optics, plasma physics, and water waves \cite{Ablowitz2011,Agrawal2003}. In one-dimensional (1D) guided-wave optics, a typical cubic NLSE application is the dynamics of optical solitons due to the tight balance of dispersion and a pure Kerr nonlinearity  \cite{Ablowitz2011,Agrawal2003}. Recently, two-dimensional (2D) solitons have attracted great attention due to their wide range applications in nonlinear 2D optical waveguide arrays and lattices \cite{Nature2003}, nonlinear photonic crystal structures \cite{Ablowitz2011, Nature2003},  Bose-Einstein condensates (BEC) \cite{Agrawal2003},
and plasma physics \cite{Horton1996}. Unlike 1D solitons, the 2D solitons can not be stabilized in {\it pure} Kerr nonlinear media \cite{Agrawal2003, Malomed2002, Malomed2011}. However, 2D solitons can exist and remain stable in a pure Kerr nonlinear optical media when subjected to an {\it external potential}. Additionally, 2D solitons can exist and also be stabilized in a modified nonlinearity medium such as in a layered structure with sign-alternating Kerr nonlinearity \cite{Malomed2002}, in a competing nonlinear medium, in a saturable nonlinear medium \cite{Weilnau2022, Malomed_2004, Torner_2010}. Lately, there has been a renewed interest in the existence and stability of 2D optical solitons in saturable nonlinear media, in order to achieve stable transmission of light beams at high speeds in 2D materials. The travelling of a single 2D soliton in saturable nonlinear media can be described by the (2+1)D nonlinear Schr\"odinger equation with saturable nonlinearity \cite{Weilnau2022, Torner_2010}. Additionally, the existence of solutions for a class of coupled saturable (2+1)D nonlinear Schr\"odinger equations has been uncovered in Ref. \cite{Maia2013} and its soliton amplitude dynamics under the effect of weak nonlinear loss has been studied in Ref. \cite{NH2021}.

Due to the significance of solitons' applications in 1D and 2D optical waveguides, several numerical schemes for NLS-type equations have been investigated, including the Crank-Nicholson scheme, the Ablowitz-Ladik scheme, the pseudo-spectral split-step method, the Hamiltonian preserving method, etc \cite{Taha1984, Yang2010}. For 1D soliton, a common approach is to use the split-step Fourier method thanks to the Fast Fourier Transform \cite{Weideman1986, Weideman1997, Taha2003, Yang2008, Bao2003}. Thereafter, there have been some finite difference schemes for the (2+1)D NLS models with pure Kerr nonlinearity and with a potential. The {\it unperturbed} (2+1)D NLS and its specific version, the Gross-Pitaevskii equation with a potential in higher dimensional spaces describing 2D and 3D solitons in Bose-Einstein condensates, has been studied in Refs. \cite{Bao2013a, Bao2013b}. In \cite{Bao2013b}, the authors analyzed finite difference methods for the unperturbed Gross-Pitaevskii equation with an angular momentum rotation term in two and three dimensions and obtained the optimal convergence rate. The methods were based on the conservative Crank-Nicolson
finite difference (CNFD) method and semi-implicit finite difference (SIFD)
method. However, as aforementioned, the 2D soliton can be stabilized in a modified nonlinearity medium and they can be {\it perturbed} by some nonlinear processes such as cubic loss due to two-photon absorption. Moreover, in spite of the important use of the saturable nonlinearity to stabilize 2D soliton, which is arising recently in 2D materials, the study on numerical schemes of (2+1)D NLS equation with saturable nonlinearity (without an external potential) {\it and with} a perturbation is still lacking. For the later perturbed (2+1)D NLS model, which includes the interplay between the saturable nonlinearity and the perturbation in a two-dimensional spatial space, the condition for stability and convergence can be changed. Therefore, it requires further analysis and a technical treatment for uniqueness and convergence of the numerical solution by the finite difference scheme.

In this paper, we establish and analyze a CNFD scheme for the (2+1)D nonlinear Schr\"odinger equation with saturable nonlinearity and a dissipation in terms of the cubic loss \cite{Weilnau2022, Torner_2010}:
\begin{align}
\left\{
\begin{array}{cc}\label{eq:NLS}
	&i\dfrac{\partial u}{\partial t}  + \Delta u + \lambda \dfrac{u|u|^2}{1+|u|^2} + i\varepsilon u |u|^2 = 0,~~\forall (x,y) \in \Omega= (a,b)\times(c,d), ~t\in(0,T),\\
	&u(x,y,0) = u_0(x,y),~~\forall (x,y) \in \Omega,\\
	&u(x,y,t) = 0,~~~~\forall (x,y) \in \Gamma = \partial \Omega, ~ t\in [0,T],
\end{array}
\right.
\end{align}
where $u(x,y,t)$ stands for the 2D spatial soliton solution with $x$, $y$, and $t$ being the spatial variables and temporal variable respectively. In equation \eqref{eq:NLS}, $\lambda$ is a real constant, $\Delta u = \dfrac{\partial^2 u}{\partial x^2} + \dfrac{\partial^2 u}{\partial y^2}$, and $\varepsilon \geq 0$ is the cubic loss coefficient. Note that in the context of nonlinear optics, the evolution variable is the propagation distance $z$ instead of the time $t$ in Eq. \eqref{eq:NLS}. The cubic loss arises in optical materials due to two-photon absorption (TPA) or gain/loss saturation in silicon media \cite{Boyd-2008}. TPA has been received considerable attention in recent years due to its importance in silicon nanowaveguides, which are expected to play a crucial role in optical processing applications in optoelectronic devices, including pulse switching and compression,
wavelength conversion, regeneration, etc. \cite{Husko2014, Loon2018, Peleg2019}. First, we establish the CNFD scheme for the perturbed (2+1)D  nonlinear Schr\"odinger equation with saturable nonlinearity. Second, we analyze the boundedness of discrete solutions and show the existence and the uniqueness of a discrete solution by the CNFD scheme. Third, we show that the convergence rate is of second order in both time step and spatial mesh sizes. Finally, we validate the numerical scheme by some numerical experiments for the propagation of 2D solitons in saturable nonlinear media.

The rest of the paper is organized as follows. In section \ref{Finite}, we present the discrete CNFD scheme for the evolution equation. In section \ref{Error}, we present the main error estimate results for the numerical scheme, including the boundedness of discrete solutions, the existence and uniqueness solution, and the analysis of convergence rate. Section \ref{Num} is reserved for the results of numerical experiments and section \ref{Con} is for the conclusion. In Appendix \ref{Appendix}, we show the consistency of the CNFD scheme.

\section{Finite difference method and main results}
\subsection{Crank–Nicolson finite difference method}
\label{Finite}
Let us consider a uniform partition with $(J+1)\times(K+1)$  points with  $(x_{j},y_{k})$ on $\Omega$, for all $ j=0,1,2,\cdots, J$ and $k=0,1,\cdots, K$. Choose the mesh sizes $\Delta x = \frac{b-a}{J}$ and $\Delta y =\frac{d-c}{K}$ and denote $h := h_{max}= \max\{\Delta x,\Delta y\}$ and $h_{min} = \min\{\Delta x, \Delta y\}$. 
One can divide the interval $[0, T]$ into $N$ sub-intervals of the constant length $\tau=\frac{T}{N}$. Then
\begin{align*}
 x_j=j\Delta x,~y_k = k\Delta y, \text{ and }t_n=n\tau, ~ \forall j = \overline{0,J};~k = \overline{0,K}; \text{ and } n = \overline{0,N}.
\end{align*}
Let $U_{j,k}^n \approx u(x_j,y_k,t_n)$ represent the numerical approximation at the grid point $(x_j,y_k,t_n)$. We define the discrete finite difference operators:
\begin{align*}
&\delta^+_t u^n_{j,k} = \frac{u^{n+1}_{j,k}-u^n_{j,k}}{\tau}, \delta_x u_{j,k}^n= \frac{u^n_{j+1,k}-u^n_{j-1,k}}{2\Delta x}, \delta_x^+ u_{j,k}^n= \frac{u^n_{j+1,k}-u^n_{j,k}}{\Delta x}, \delta_x^- u_{j,k}^n= \frac{u^n_{j,k}-u^n_{j-1,k}}{\Delta x},\\
&\delta_t u^n_{j,k} = \frac{u^{n+1}_{j,k}-u^{n-1}_{j,k}}{2\tau}, \delta_y u_{j,k}^n= \frac{u^n_{j,k+1}-u^n_{j,k-1}}{2\Delta y}, \delta_y^+ u_{j,k}^n= \frac{u^n_{j,k+1}-u^n_{j,k}}{\Delta y},  \delta_y^- u_{j,k}^n= \frac{u^n_{j,k}-u^n_{j,k-1}}{\Delta y},\\
&\delta^2_{x} u_{j,k}^n  = \frac{u_{j+1,k}^n-2u_{j,k}^n+u_{j-1,k}^n}{(\Delta x)^2}, ~~\delta^2_{y} u_{j,k}^n  = \frac{u_{j,k+1}^n-2u_{j,k}^n+u_{j,k-1}^n}{(\Delta y)^2},\\
&\delta u^n_{j,k} = (\delta_x u_{j,k}^n,\delta_y u_{j,k}^n),~~\delta^+ u^n_{j,k} = (\delta_x^+ u_{j,k}^n,\delta_y^+ u_{j,k}^n),~~\delta^2 u^n_{j,k} =  \delta^2_x u^n_{j,k} + \delta^2_y u^n_{j,k}.
\end{align*}
Let
\begin{align*}
&\T_{JK} = \big\{(j,k)| j =1,2,\cdots,J-1 \text { and } k=1,2,\cdots,K-1\big\},\\
&\T^0_{JK} = \big\{(j,k)| j =0,1,\cdots,J  \text { and } k=0,1,\cdots,K\big\},\\
&X_{JK}  = \left\{u = (u_{j,k})_{(j,k)\in \T^0_{JK}}|  u_{0,k} = u_{J,k} = u_{j,0}=u_{j,K} = 0,~\forall (j,k)\in \T^0_{JK}\right\},
\end{align*}
equipped with inner products and norms as
\begin{align*}
&(u,v)_h = \Delta x\Delta y\sum_{j=0}^{J-1}\sum_{k=0}^{K-1}u_{j,k}\bar{v}_{j,k},~\|u\|^2_{2,h} = ( u,u)_h,~\|u\|^p_{p,h} = \Delta x \Delta y\sum_{j=0}^{J-1}\sum_{k=0}^{K-1}|u_{j,k}|^p,\\
& \langle u,v\rangle_h = \Delta x\Delta y\sum_{j=1}^{J-1}\sum_{k=1}^{K-1}u_{j,k}\bar{v}_{j,k}, |||u|||^2_{2,h} = \langle u,u\rangle_h,  ~ \|u\|_{\infty,h} = \sup_{(j,k)\in\T^0_{JK}}|u_{j,k}|,\\
&|u|^2_{1,h}=\|\delta^+u\|^2_{2,h}= \|\delta_x^+ u\|^2_{2,h} + \|\delta_y^+ u\|^2_{2,h}.
\end{align*}
One has the following discrete equalities and inequalities for $u,v\in X_{JK}$:
\begin{align}
&\langle \delta_x u,v \rangle_h =- \langle  u, \delta_x v \rangle_h,~~~~\langle \delta^2_x u,v \rangle_h = - ( \delta_x^+ u,\delta^+_x v)_h\label{discretegreen_x},\\
&\langle \delta_y u,v \rangle_h =- \langle u, \delta_y v \rangle_h,~~~~\langle \delta^2_y u,v \rangle_h = - ( \delta_y^+ u,\delta^+_y v)_h\label{discretegreen_y},\\
&\|u\|_{2,h}^2\leq \|\delta^+u\|^2_{2,h}, ~~~\|u\|^4_{4,h}\leq \|u\|_{2,h}^2\|\delta^+u\|^2_{2,h}\label{ineq:poincare1}.
\end{align}
The readers can see the proofs of the relations above in \cite{Bao2013a}.\\
Now, let $f(s)=\frac{s}{1+s}$ and  $F(\rho)=\int_0^\rho f(s)ds=\rho-\ln(|1+\rho|)$. The CNFD discretization of equation \eqref{eq:NLS} can be written in the following form:
\begin{align}\label{CNFD}
\left\{
\begin{array}{c}
i\delta_t^+U^n_{j,k}+\delta^2 U_{j,k}^{n+1/2} + \lambda \psi(U^{n+1}_{j,k},U^n_{j,k}) +i\varepsilon \varphi(U^{n+1}_{j,k},U^n_{j,k}) =0, ~~\forall (j,k) \in \T_{JK},\\
U^0_{j,k} = u_0(x_j,y_k), ~\forall (j,k) \in \T^0_{JK}, ~ U^{n+1}\in X_{JK},
\end{array}
\right.
\end{align}
where $U^{n+1/2}_{j,k} = \frac{U^{n+1}_{j,k}+U^n_{j,k}}{2}$,
$\psi(z,w)=\frac{F(|z|^2)-F(|w|^2)}{|z|^2-|w|^2}\frac{z+w}{2}$ or
\begin{align*} 												     \psi(z,w)=\frac{z+w}{2}\int^1_0f(|w|^2+t(|z|^2-|w|^2))dt,
\end{align*}
and
\begin{align*}
	\varphi(z,w) = |\frac{z+w}{2}|^2\frac{z+w}{2}.
\end{align*}
Additionally, one can define the discrete energies (or Hamiltonian):
\begin{align}
	E_h(U^n) &= |U^{n}|^2_{1,h} -\lambda \|F(|U^n|^2)\|_{1,h},\label{df:E_h}\\
	\E_h(U^n) &= |U^{n}|^2_{1,h}\label{df:E_h1}.
\end{align}
The settings above play a crucial role in our analysis and estimations for the existence and uniqueness of the numerical solution and for proving the convergence of the CNFD scheme. 
\subsection{Main results}
From now on, without loss of generality, we assume $\Delta x = \Delta y = h$. The main results of the paper are the following two theorems.
\begin{thm}[Unique solution]\label{thm:existsol_uniqueSolution_CNFD}
	For any given $U^n$, there	exists a unique solution $U^{n+1}$ of the CNFD discretization in (\ref{CNFD}).
\end{thm}
Theorem \ref{thm:existsol_uniqueSolution_CNFD} is proved in Subsection \ref{subsect_exist_unique_sol}.
\begin{thm}[Convergence]\label{thm:err_uandU_L2H01norm}
	Let the solution of (\ref{eq:NLS}) be smooth enough and $\{U^n\}_{n=0}^{N}$ satisfy (\ref{CNFD}). For $\tau\lesssim h $, there exists a constant $c > 0$ which is independent of $h$ and $\tau$ such that
	\begin{align}
        \max_{n\in [1,N]}\|u^n-U^n\|_{2,h} \leq c(\tau^2+h^2),\label{ineq:errorUandu_L2}\\
		\max_{n\in [1,N]}|u^n-U^n|_{1,h} \leq c(\tau^2+h^2).\label{ineq:errorUandu_H1}
	\end{align}
\end{thm}
The proof of Theorem \ref{thm:err_uandU_L2H01norm} is presented in Subsection \ref{subsect_convergence}.
To prove the theorems above, we first establish some lemmas for the boundedness of discrete solutions, stability, and error estimates.
\section{Error Estimates}
\label{Error}
Noting that for a solution $u(x,y,t)$ of \eqref{eq:NLS}, it can be shown that (see for example, equation (6) in Ref. \cite{NH2021}):
\begin{align*}
\|u(\cdot,t)\|_{2} \leq\|u_0\|_{2}, \forall t\in (0,T).
\end{align*}
 First, we study the boundedness of the discrete solutions and of the discrete energies.
\subsection{Boundedness of discrete solutions}
\begin{lm}[Boundedness of discrete solutions]\label{lm:discreteConservativeMass}
	Given $\{U^n\}_{n=0,1,\cdots,N}$ be the solution in the CNFD scheme (\ref{CNFD}),  then
	\begin{align}
		\|U^n\|^2_{2,h} \leq \|U^0\|^2_{2,h}\label{discreteConservativeMass}.
	\end{align}
\end{lm}
\begin{proof}
    Multiplying the first term in the scheme \eqref{CNFD} with $\overline{U^{n+1/2}_{j,k}}$ and then summing over all  $j=1,\cdots,J-1$ and $k=1,\cdots,K-1$, it yields
	\begin{align*}
		i\langle\delta_t^+U^n,U^{n+1/2}\rangle_h = i\frac{\|U^{n+1}\|^2_{2,h}-\|U^n\|^2_{2,h}}{2\tau} + \frac{i}{2\tau}\sum_{k=1}^{K-1}\sum_{j=1}^{J-1}(U^{n+1}_{j,k}\overline{U^n_{j,k}} - U^{n}_{j,k}\overline{U^{n+1}_{j,k}}).
	\end{align*}
Since
	$$
\Re\left(U^{n+1}_{j,k}\overline{U^n_{j,k}} - U^{n}_{j,k}\overline{U^{n+1}_{j,k}}\right) = 0,
	$$
then
\begin{align}\label{impartdt}
\Im\left(i\langle\delta_t^+U^n,U^{n+1/2}\rangle_h\right) =\frac{\|U^{n+1}\|^2_{2,h}-\|U^n\|^2_{2,h}}{2\tau}.
	\end{align}
Using (\ref{discretegreen_x}) and (\ref{discretegreen_y}), it implies
	\begin{align}
		\langle\delta^2_x U^{n+1/2},U^{n+1/2}\rangle_h = -(\delta^+_x U^{n+1/2},\delta^+_xU^{n+1/2})_h\label{impartdiffusion_x},\\
		\langle\delta^2_y U^{n+1/2},U^{n+1/2}\rangle_h = -(\delta^+_y U^{n+1/2},\delta^+_y U^{n+1/2})_h\label{impartdiffusion_y}.
	\end{align}
We take in (\ref{CNFD}) the inner product with $U^{n+1/2}$. By using (\ref{impartdt}),  (\ref{impartdiffusion_x}), and (\ref{impartdiffusion_y})  and taking imaginary parts, it yields
	\begin{align*}
		\frac{\|U^{n+1}\|^2_{2,h}- \|U^{n}\|^2_{2,h}}{2\tau} + \varepsilon\sum_{j=1}^{J-1}\sum_{k=1}^{K-1}|U^{n+1/2}_{j,k}|^4=0.
	\end{align*}
One can estimate the term $\|U^{n+1}\|^2_{2,h}$ as
	\begin{align*}
		\|U^{n+1}\|^2_{2,h} \leq \|U^{n}\|^2_{2,h}.
	\end{align*}
This arrives at
  \begin{align*}
\|U^{n}\|^2_{2,h} \leq \|U^{0}\|^2_{2,h}.
	\end{align*} 
\end{proof}
We note that when $\varepsilon=0$, one has the conservation $\|U^{n}\|^2_{2,h} = \|U^{0}\|^2_{2,h}$, for all $n=0,1,\cdots,N$. 
Next, the discrete $H^1(\Omega)$-norm of the discrete solution $U^n$ is estimated by the following lemma.
\begin{lm}\label{lm:bound_gradient_byEnergy}
Given $U^m\in X_{JK}$, then
\begin{align*}
	|U^m|^2_{1,h} \leq E_h(U^m) + |\lambda|\|U^m\|^2_{2,h}.
\end{align*}
\end{lm}
\begin{proof}The proof is divided two parts.\\ 
{\bf  Part I:  $\lambda \leq 0$}. From the definition of the discrete energy in \eqref{df:E_h}, it implies
\begin{align*}
	E_h(U^m) &= |U^{m}|^2_{1,h}  -\lambda\|F(|U^{m}|^2)\|_{1,h} \geq |U^{m}|^2_{1,h}.
\end{align*}
Thus,
\begin{align*}
	|U^m|^2_{1,h} \leq E_h(U^m) + |\lambda|\|U^m\|^2_{2,h}.
\end{align*}
{\bf Part II:} $\lambda > 0$. From the definition of the discrete energy in \eqref{df:E_h}, one gets
\begin{align*}
	E_h(U^m) &= |U^{m}|^2_{1,h}  -\lambda\|F(|U^{m}|^2)\|_{1,h}\\
	& \geq |U^m|^2_{1,h} -  \lambda \|U^m\|^2_{2,h} +\lambda\|\ln(1+|U^m|^2)\|_{1,h}.
\end{align*}
Thus,
\begin{align*}
	|U^m|^2_{1,h}  \leq E(U^m) + \lambda \|U^m\|^2_{2,h}.
\end{align*}
Via the results of {\bf Part I} and {\bf Part II}, the proof is completed. 
\end{proof}

\begin{lm}\label{lm:est:Im_Un+1_Un_U_n2}
Given $\{U^n\}_{n=0}^{N}$ be the solution in the scheme \eqref{CNFD}, then
    \begin{align}\label{est:Im_Un+1_Un_U_n2}
\Im\left(\langle U^{n+1},{U^n}|U^{n+1/2}|^2\rangle_h\right)
\leq 
		\tau\lambda\langle \frac{F(|U^{n+1}|^2)-F(|U^n|^2)}{|U^{n+1}|^2-|U^n|^2}|U^{n+1/2}|^2,|U^{n+1/2}|^2\rangle_h,\nonumber \\
  ~\forall n=0,1,\cdots,N-1.
	\end{align}
\end{lm}
\begin{proof}
From (\ref{CNFD}), one gets the following expression:
	\begin{align*}
		i\delta_t^+U^n_{j,k}\overline{U^{n+1/2}_{j,k}} = i\frac{|U^{n+1}_{j,k}|^2-|U^n_{j,k}|^2}{2\tau} + \frac{i}{2\tau}(U^{n+1}_{j,k}\overline{U^n_{j,k}} - U^{n}_{j,k}\overline{U^{n+1}_{j,k}})\\
		=i\frac{|U^{n+1}_{j,k}|^2-|U^n_{j,k}|^2}{2\tau} - \frac{1}{\tau}\Im(U^{n+1}_{j,k}\overline{U^n_{j,k}} ).
	\end{align*}
Taking in the previous relation the inner product with $|U^{n+1/2}|^2$ yields
	\begin{align}
		i\langle\delta_t^+U^n_{j,k},{U^{n+1/2}}|U^{n+1/2}|^2\rangle_h = \frac{i}{2\tau}\langle |U^{n+1}|^2-|U^n|^2,|U^{n+1/2}|^2\rangle_h
		\nonumber\\- \frac{1}{\tau}\Im(\langle U^{n+1},{U^n}|U^{n+1/2}|^2\rangle_h).\label{equa:delta_t_U^2}
	\end{align}
	Using (\ref{discretegreen_x}) and (\ref{discretegreen_y}), there hold
	\begin{align}
		\langle\delta^2_x U^{n+1/2},U^{n+1/2}|U^{n+1/2}|^2\rangle_h = -(\delta^+_x U^{n+1/2},\delta^+_x(U^{n+1/2}|U^{n+1/2}|^2))_h\label{equa:green_delta_x_U^2},\\
		\langle\delta^2_y U^{n+1/2},U^{n+1/2}|U^{n+1/2}|^2\rangle_h = -(\delta^+_y U^{n+1/2},\delta^+_y(U^{n+1/2}|U^{n+1/2}|^2))_h\label{equa:green_delta_y_U^2}.
	\end{align}
	From the definition of the 
 inner product $(\cdot,\cdot)_h$ and the discrete gradient, one has
\begin{eqnarray*}
\!\!\!\!\!\!\!\!\!\!\!\!\!\!\!\!\!\!\!
L_1&&= (\delta^+_x U^{n+1/2},\delta^+_x(U^{n+1/2}|U^{n+1/2}|^2))_h
 \nonumber \\&&
= \Delta x \Delta y\sum_{j=0}^{J-1}\sum_{k=0}^{K-1}\frac{U^{n+1/2}_{j+1,k}-U^{n+1/2}_{j,k}}{\Delta x}\frac{\overline{U^{n+1/2}_{j+1,k}}|U^{n+1/2}_{j+1,k}|^2-\overline{U^{n+1/2}_{j,k}}|U^{n+1/2}_{j,k}|^2}{\Delta x}.
\end{eqnarray*}
	Moreover 
	\begin{align*}		\overline{U^{n+1/2}_{j+1,k}}|U^{n+1/2}_{j+1,k}|^2-\overline{U^{n+1/2}_{j,k}}|U^{n+1/2}_{j,k}|^2 = (\overline{U^{n+1/2}_{j+1,k}} - \overline{U^{n+1/2}_{j,k}})\frac{|U^{n+1/2}_{j+1,k}|^2+|U^{n+1/2}_{j,k}|^2}{2}\nonumber\\ + \frac{\overline{U^{n+1/2}_{j+1,k}} + \overline{U^{n+1/2}_{j,k}}}{2}(|U^{n+1/2}_{j+1,k}|^2-|U^{n+1/2}_{j,k}|^2).
	\end{align*}
	Therefore,
	\begin{eqnarray*}
\!\!\!\!\!\!\!\!\!\!\!	
		L_1&&
		= \Delta x \Delta y\sum_{j=0}^{J-1}\sum_{k=0}^{K-1}\frac{U^{n+1/2}_{j+1,k}-U^{n+1/2}_{j,k}}{\Delta x}\frac{\overline{U^{n+1/2}_{j+1,k}} - \overline{U^{n+1/2}_{j,k}}}{\Delta x}\frac{|U^{n+1/2}_{j+1,k}|^2+|U^{n+1/2}_{j,k}|^2}{2}\nonumber \\&&	
		 +\Delta x \Delta y\sum_{j=0}^{J-1}\sum_{k=0}^{K-1}\frac{U^{n+1/2}_{j+1,k}-U^{n+1/2}_{j,k}}{\Delta x}\frac{\overline{U^{n+1/2}_{j+1,k}}+\overline{U^{n+1/2}_{j,k}}}{2}\frac{|U^{n+1/2}_{j+1,k}|^2-|U^{n+1/2}_{j,k}|^2}{\Delta x}.
	\end{eqnarray*}
	Taking the real part for the previous relation, it yields 
\begin{eqnarray*}
\!\!\!\!\!\!\!	
		\Re(L_1)&&	
		= \frac{\Delta x \Delta y}{2}\sum_{j=0}^{J-1}\sum_{k=0}^{K-1}\bigg|\frac{U^{n+1/2}_{j+1,k}-U^{n+1/2}_{j,k}}{\Delta x}\bigg|^2(|U^{n+1/2}_{j+1,k}|^2+|U^{n+1/2}_{j,k}|^2)
\nonumber \\&&
+\frac{\Delta x \Delta y}{2} \sum_{j=0}^{J-1}\sum_{k=0}^{K-1}\bigg|\frac{|U^{n+1/2}_{j+1,k}|^2-|U^{n+1/2}_{j,k}|^2}{\Delta x}\bigg|^2
\!\!\!\!\!\!\!\!\!\!\!\!\!\!\!\!\!\!\!
\nonumber \\&&
		 =\frac{\Delta x \Delta y}{2}\sum_{j=0}^{J-1}\sum_{k=0}^{K-1}\bigg|\frac{U^{n+1/2}_{j+1,k}-U^{n+1/2}_{j,k}}{\Delta x}\bigg|^2(|U^{n+1/2}_{j+1,k}|^2+|U^{n+1/2}_{j,k}|^2)
		\nonumber\\
  &&+	\frac{1}{2}(\delta^+_x |U^{n+1/2}|^2,\delta_x^+|U^{n+1/2}|^2)_h.
	\end{eqnarray*}
It is easy to see that $\Re(L_1)\geq 0$. From \eqref{equa:green_delta_x_U^2} and the definition of $L_1$, one then obtains
\begin{align}
	\Re(\langle\delta^2_x U^{n+1/2},U^{n+1/2}|U^{n+1/2}|^2\rangle_h) \leq 0\label{inequa:Re_delta_x_U^2}.
\end{align}
By using \eqref{equa:green_delta_y_U^2} and implementing the similar calculations for respect $y$, it arrives at
\begin{align}
	\Re(\langle\delta^2_y U^{n+1/2},U^{n+1/2}|U^{n+1/2}|^2\rangle_h) \leq 0 \label{inequa:Re_delta_y_U^2}.
\end{align}
We	take in (\ref{CNFD}) the inner product with $U^{n+1/2}|U^{n+1/2}|^2$ and then take the real part. By  using \eqref{equa:delta_t_U^2}, \eqref{inequa:Re_delta_x_U^2} and \eqref{inequa:Re_delta_y_U^2}, it yields
	\begin{align*}
		- \frac{1}{\tau}\Im(\langle U^{n+1},{U^n}|U^{n+1/2}|^2\rangle_h)  + \lambda\langle \frac{F(|U^{n+1}|^2)-F(|U^n|^2)}{|U^{n+1}|^2-|U^n|^2}|U^{n+1/2}|^2,|U^{n+1/2}|^2\rangle_h \geq 0.
	\end{align*}
One then obtains
	\begin{align*}
		\Im(\langle U^{n+1},{U^n}|U^{n+1/2}|^2\rangle_h)\leq 
		\tau\lambda\langle \frac{F(|U^{n+1}|^2)-F(|U^n|^2)}{|U^{n+1}|^2-|U^n|^2}|U^{n+1/2}|^2,|U^{n+1/2}|^2\rangle_h.
	\end{align*}
	This completes the proof. 
\end{proof}
\begin{lm}[{Boundedness of energy}]\label{lm:EUnbounded}
	Given $\{U^n\}_{n=0}^{N}$ be the solution in the scheme \eqref{CNFD}, then $E(U^{n+1})$ is bounded for all $n=0,1,\cdots,N-1$.
\end{lm}
\begin{proof}
	We note that
	\begin{align}\label{realpartdt}
		i\langle\delta_t^+U^n,U^{n+1}-U^n\rangle_h = i\frac{\|U^{n+1}-U^n\|^2_{2,h}}{\tau}.
	\end{align}
	Using (\ref{discretegreen_x}) and (\ref{discretegreen_y}), there holds 
	\begin{align}\nonumber
		\begin{split}
			\langle\delta^2_x U^{n+1/2},U^{n+1}-U^n\rangle_h + \langle\delta^2_y U^{n+1/2},U^{n+1}-U^n\rangle_h\\
			= -(\delta^+_x U^{n+1/2},\delta^+_x(U^{n+1}-U^n))_h -(\delta^+_y U^{n+1/2},\delta^+_y (U^{n+1}-U^n))_h.
		\end{split}
	\end{align}
 Then,
 	\begin{align}
			\langle\delta^2_x U^{n+1/2},U^{n+1}-U^n\rangle_h + \langle\delta^2_y U^{n+1/2},U^{n+1}-U^n\rangle_h \nonumber\\
			= -\frac{1}{2}|U^{n+1}|^2_{1,h} + \frac{1}{2}|U^n|^2_{1,h} + \frac{1}{2}(\delta^+_x U^{n+1},\delta^+_xU^n)_h - \frac{1}{2}(\delta^+_x U^{n},\delta^+_xU^{n+1})_h \nonumber\\ + \frac{1}{2}(\delta^+_y U^{n+1},\delta^+_yU^n)_h - \frac{1}{2}(\delta^+_y U^{n},\delta^+_y U^{n+1})_h.\label{realpartdiffusion}
	\end{align}
	For the discrete cubic loss, it yields 
	\begin{align}\label{loss_cubic}
		i\langle U^{n+1/2}|U^{n+1/2}|^2,U^{n+1}-U^n \rangle_h = \frac{i}{2}\langle |U^{n+1/2}|^2,|U^{n+1}|^2-|U^n|^2\rangle_h\nonumber\\+\Im(\langle U^{n+1},U^n|U^{n+1/2}|^2\rangle_h).
	\end{align}	
	Taking in (\ref{CNFD}) the inner product with $U^{n+1}-U^n$, using (\ref{realpartdt}), (\ref{realpartdiffusion}) and \eqref{loss_cubic}, and taking real parts, one gets
	\begin{align*}
		&|U^{n+1}|^2_{1,h}  -\lambda\|F (|U^{n+1}|^2)\|_{1,h}- 2\varepsilon\Im(\langle U^{n+1},U^n|U^{n+1/2}|^2\rangle_h)
		= |U^{n}|^2_{1,h}   -\lambda\|F(|U^{n}|^2)\|_{1,h}.
	\end{align*}
	Using \eqref{est:Im_Un+1_Un_U_n2} in Lemma \ref{lm:est:Im_Un+1_Un_U_n2}, one can arrive at
	\begin{align*}
		E_h(U^{n+1})\leq E_h(U^{n})
		+2\lambda\varepsilon\tau\langle \frac{F(|U^{n+1}|^2)-F(|U^n|^2)}{|U^{n+1}|^2-|U^n|^2}|U^{n+1/2}|^2,|U^{n+1/2}|^2\rangle_h.
	\end{align*}
	Using the property of the function $F$, one gets
	\begin{align*}
		E_h(U^{n+1}) -E_h(U^n) \leq 2|\lambda|\tau\varepsilon\|U^{n+1/2}\|^4_{4,h}\leq |\lambda|\varepsilon\tau (\|U^{n+1}\|^4_{4,h}+\|U^{n}\|^4_{4,h}).
	\end{align*} 
	Applying inequality \eqref{ineq:poincare1}, Lemma \ref{lm:discreteConservativeMass}, and Lemma \ref{lm:bound_gradient_byEnergy}, there holds
	\begin{align*}
		E_h(U^{n+1}) -E_h(U^n) &\leq |\lambda|\varepsilon\tau\|U^0\|^2_{2,h} \left[E(U^{n+1})+E(U^n)+2|\lambda|\|U^0\|^2_{2,h}\right].
	\end{align*} 
Thus, 
	\begin{align*}
		E_h(U^{n+1}) -E_h(U^n) &\leq 2|\lambda|^2\varepsilon\tau\|U^0\|^4_{2,h} +  |\lambda|\varepsilon\tau\|U^0\|^2_{2,h} \left[E(U^{n+1})+E(U^n)\right].
	\end{align*} 
	By using Gr\"onwall's inequality, one obtains
	\begin{align*}
		E_h(U^{n+1}) \lesssim \|U^0\|^4_{2,h} + E_h(U^0).
	\end{align*}
	This means that $E_h(U^{n+1})$ is bounded for all $n=0,1,\cdots,N-1$. 
\end{proof}

\subsection{Existence and Uniqueness Solution}\label{subsect_exist_unique_sol}
In this section, we prove the existence and the uniqueness of the solution of the CNFD scheme. To show the existence of a solution, one can use the following well-known Brouwer-type fixed point theorem.
\begin{lm}\label{lm:existencesolution}
	Let $H(\cdot,\cdot)$ be a finite dimensional inner product space, $\|\cdot\|$ be the associated norm, and $g: H\rightarrow H$ be continuous. Assume moreover that
	$$
	\exists \alpha > 0 ~\forall z \in H ~ \|z\| = \alpha,~ \Re(g(z),z) \geq 0.
	$$
	Then, there exists a $z^*\in H$ such that $g(z^*) = 0$ and $\|z^*\|\leq \alpha$.
\end{lm}
{\it Proof of Theorem \ref{thm:existsol_uniqueSolution_CNFD}}. 
From (\ref{CNFD}), one can get for all $(j,k)\in\T_{JK}$ that
\begin{align*}
U^{n+1/2}_{j,k} = U^n_{j,k} +\frac{i\tau}{2}\delta^2 U^{n+1/2}_{j,k} + \frac{i\tau\lambda}{2}\frac{F(|2U^{n+1/2}_{j,k}-U^n_{j,k}|^2)-F(|U^n_{j,k}|^2)}{|2U^{n+1/2}_{j,k}-U^n_{j,k}|^2-|U^n_{j,k}|^2}U^{n+1/2}_{j,k}\\
-\frac{\tau\varepsilon}{2} |U^{n+1/2}_{j,k}|^2 U^{n+1/2}_{j,k}.
\end{align*}
Note that the mapping $\Pi: X_{JK}\rightarrow X_{JK}$ by
\begin{align*}
\begin{split}
(\Pi(v))_{j,k} = v_{j,k}-U^n_{j,k}  -  \frac{i\tau}{2}\delta^2 v_{j,k} - \frac{i\tau\lambda}{2}\frac{F(|2v_{j,k}-U^n_{j,k}|^2)-F(|U^n_{j,k}|^2)}{|v_{j,k}-U^n_{j,k}|^2-|U^n_{j,k}|^2}v_{j,k} + \frac{\tau\varepsilon}{2} |v_{j,k}|^2  v_{j,k}
\end{split}
\end{align*}
is continuous. We then have $\Re\langle\Pi(v),v\rangle_h \geq \|v\|^2_{2,h}-\Re(\langle U^n,v\rangle_h)$, i.e 
\begin{align*}
\Re\langle\Pi(v),v\rangle_h \geq \|v\|_{2,h}(\|v\|_{2,h}-\|U^n\|_{2,h}).
\end{align*}
Hence, for $\|v\|_{2,h} = \|U^n\|_{2,h}+1$, it yields $\Re\langle\Pi(v),v\rangle_h > 0$. Thus, the existence of $U^{n+1/2}$ follows from Lemma \ref{lm:existencesolution}.\\
Now we will prove the unique solution. Let $v,w\in X_{JK}$ be the solutions in the scheme \eqref{CNFD} such that $\Pi(v) = \Pi(w) = 0$. Setting $\chi = v-w$, it yields
	\begin{align}\label{unique_CNFD}
		\chi_{j,k} =\frac{i\tau}{2}\delta^2 \chi_{j,k} + \frac{i\lambda\tau}{2}\psi_1(v_{j,k},w_{j,k}) + \frac{\tau\varepsilon}{2}\psi_2(v_{j,k},w_{j,k}),
	\end{align}
	where 
 $$\psi_1(v_{j,k},w_{j,k}) = \frac{F(|2v_{j,k}-U^n_{j,k}|^2)-F(|U^n_{j,k}|^2)}{|2v_{j,k}-U^n_{j,k}|^2-|U^n_{j,k}|^2}v_{j,k}-\frac{F(|2w_{j,k}-U^n_{j,k}|^2)-F(|U^n_{j,k}|^2)}{|2w_{j,k}-U^n_{j,k}|^2-|U^n_{j,k}|^2}w_{j,k}$$
 and 
$$\psi_2(v_{j,k},w_{j,k})=        |v_{j,k}|^2v_{j,k}-|w_{j,k}|^2w_{j,k}.$$
Setting $v_1 = |2v_{j,k}-U^n_{j,k}|^2-|U^{n}_{j,k}|^2$ and $w_1 = |2w_{j,k}-U^n_{j,k}|^2-|U^{n}_{j,k}|^2$,
one obtains 
	\begin{align*}
		\psi_1(v_{j,k},w_{j,k})& = v_{j,k}\int_0^1f(|U^n_{j,k}|^2+tv_1)dt - w_{j,k}\int_0^1f(|U^n_{j,k}|^2+tw_1)dt\\
&=v_{j,k}\bigg[\int_0^1f(|U^n_{j,k}|^2+tv_1)dt -\int_0^1f(|U^n_{j,k}|^2+tw_1)dt\bigg]\\
		&+ (v_{j,k}-w_{j,k})\int_0^1f(|U^n_{j,k}|^2+tw_1)dt.
	\end{align*}
 Then,
 	\begin{align*}
		\psi_1(v_{j,k},w_{j,k}) 
		=v_{j,k}\int_0^1t(v_1-w_1)\int_0^1f'(|U^n_{j,k}|^2+tw_1+ts(v_1-w_1))dsdt\\
		+\chi_{j,k}\int_0^1f(|U^n_{j,k}|^2+tw_1)dt.
	\end{align*}
Since $|f(s)|\leq 1$ and $|f'(s)|\leq 1$ for all $s\geq 0$, then
	\begin{align*}
		|\psi_1(v_{j,k},w_{j,k})|&\leq \frac{1}{2}|v_{j,k}||v_1-w_1| +|\chi_{j,k}|
		\leq 4|v_{j,k}|(|v_{j,k}|+|w_{j,k}|+|U^n_{j,k}|)|\chi_{j,k}|+|\chi_{j,k}|\\
		&\leq 2(4|v_{j,k}|^2+|w_{j,k}|^2+|U^n_{j,k}|^2)|\chi_{j,k}|+|\chi_{j,k}|
	\end{align*}
and
	\begin{align*}
		|\psi_2(v_{j,k},w_{j,k})|&\leq |v_{j,k}|^2|\chi_{j,k}|+||v_{j,k}|^2-|w_{j,k}|^2||w_{j,k}|
		\leq |v_{j,k}|^2|\chi_{j,k}| +(|v_{j,k}|+|w_{j,k}|)|\chi_{j,k}||w_{j,k}|\\
		&\leq |\chi_{j,k}|(|v_{j,k}|^2+|v_{j,k}||w_{j,k}| + |w_{j,k}|^2)
		\leq \frac{3}{2}(|v_{j,k}|^2+ |w_{j,k}|^2)|\chi_{j,k}|.
	\end{align*}
 Taking in \eqref{unique_CNFD} the inner product with $\chi$, taking the real and imaginary parts, respectively, and then using Holder’s inequality in the right-hand sides of the resulting identities, one can get
	\begin{align}
		\|\chi\|_{2,h}^2 \leq \tau|\lambda|(4\|v\|^2_{4,h} +\|w\|^2_{4,h} + \|U^n\|^2_{4,h})\|\chi\|^2_{4,h} +\frac{|\lambda|\tau}{2}\|\chi\|^2_{2,h} \nonumber\\+\frac{3\varepsilon\tau}{4}(\|v\|^2_{4,h}+\|w\|^2_{4,h}) \|\chi\|^2_{4,h}\label{ineq:est_norml2chi}
	\end{align}
and 
	\begin{align}
		|\chi|_{1,h}^2 &\leq |\lambda|(4\|v\|^2_{4,h} +\|w\|^2_{4,h} + \|U^n\|^2_{4,h})\|\chi\|^2_{4,h} + \frac{3\varepsilon}{2}(\|v\|^2_{4,h}+\|w\|^2_{4,h}) \|\chi\|^2_{4,h}.\label{ineq:est_normh1chi}
	\end{align}
From Lemma \ref{lm:EUnbounded}, Lemma \ref{lm:discreteConservativeMass}, and Lemma \ref{lm:bound_gradient_byEnergy}, for $m\in \overline{0,N}$, there exists a constant $c>0$ such that 
	\begin{align*}
		|U^m|^2_{1,h}\leq E_h(U^m)+|\lambda|\|U^m\|^2_{2,h} \leq  cE_h(U^0) + c\|U^0\|^4_{2,h} + |\lambda||\|U^0\|^2_{2,h}.
	\end{align*}
	Applying  \eqref{ineq:poincare1}, \eqref{discreteConservativeMass}, and previous inequality, there holds
	\begin{align*}
		\|U^m\|^4_{4,h}\leq |U^m|^2_{1,h}\|U^m\|^2_{2,h}\leq (cE_h(U^0) + c\|U^0\|^4_{2,h} + |\lambda||\|U^0\|^2_{2,h})\|U^0\|^2_{2,h}.
	\end{align*}
	From the previous inequality, it implies that $\|v\|^4_{4,h}$, $\|w\|^4_{4,h}$, and $\|U^n\|^4_{4,h}$ are bounded.  
	Then, the inequalities (\ref{ineq:est_norml2chi}) and (\ref{ineq:est_normh1chi}) become, respectively, as
	\begin{align}
		&\|\chi\|_{2,h}^2 \leq c_1\tau\|\chi\|^2_{4,h}\label{ineq:est_norml2chi1}
	\end{align}
and 
	\begin{align}
		&|\chi|_{1,h}^2 \leq c_2\|\chi\|^2_{4,h}.\label{ineq:est_normh1chi2}
	\end{align}
Using (\ref{ineq:est_norml2chi1}) and (\ref{ineq:est_normh1chi2}), it arrives at 
	\begin{align*}
		\|\chi\|^4_{4,h}\leq c_1c_2\tau\|\chi\|^4_{4,h}.
	\end{align*}
Thus, for $\tau$ is small enough, the proof of the unique solution is completed. \qed
\subsection{Convergence}\label{subsect_convergence}
Setting $M:=\max\{|u(\xx,t)|:(\xx,t)\in\Omega\times(0,T)\}+1$, we define the auxiliary functions $\tilde{\varphi}:\mathbb{C}\times\mathbb{C}\rightarrow \mathbb{C}$
\begin{align*}
	\tilde{\varphi}(z,w):=\left\{\begin{array}{cc}
		\varphi(z,w) & \text{if }|z+w| \leq 2M,\\
		M^2\frac{(z+w)}{2} & \text{if }|z+w| > 2M.
	\end{array}
	\right.
\end{align*}
and  $\tilde{\psi}:\mathbb{C}\times\mathbb{C}\rightarrow \mathbb{C}$
\begin{align*}
	\tilde{\psi}(z,w):=\left\{\begin{array}{cc}
		\psi(z,w) & \text{if }|z|,|w| \leq M,\\
		\frac{F(M^2)-F(|w|^2)}{M^2-|w|^2}\frac{(z+w)}{2} & \text{if }|z|> M,|w| \leq M,\\
		\frac{F(|z|^2)-F(M^2)}{|z|^2-M^2}\frac{(z+w)}{2} & \text{if }|z|\leq M,|w| > M,\\
		f(M^2)\frac{(z+w)}{2} & \text{if }|z|> M,|w| >  M.
	\end{array}
	\right.
\end{align*}
We note that the functions $\tilde{\varphi}$ and $\tilde{\psi}$ are globally Lipschitz continuous. Let $V^0:=U^0$ and $V^n\in X_{JK}$, $n=1,\cdots,N$ satisfy
\begin{align}\label{soltuonV_globallylipschtiz}
	\left\{
	\begin{array}{c}
		i\delta_t^+V^n_{j,k}+\delta^2 V_{j,k}^{n+1/2} +\lambda\tilde{\psi}(V^{n+1}_{j,k},V^n_{j,k}) +i\varepsilon \tilde{\varphi}(V^{n+1}_{j,k},V^n_{j,k})=0, ~~\forall (j,k) \in \T_{JK},\\
		V^0_{j,k} = u_0(x_j,y_k), ~\forall (j,k) \in \T^0_{JK}, V^{n+1}\in X_{JK}.
	\end{array}
	\right.
\end{align}
From the definition of the discrete solution $\{V^n\}_{n=0}^n$, one can get the following lemma.
\begin{lm}
	Let the solution of (\ref{eq:NLS}) be smooth enough and $\{V^n\}_{n=0}^N$ satisfy (\ref{soltuonV_globallylipschtiz}). For $\tau\lesssim h $, then
	\begin{align}\label{ineq:errorVandu}
		\max_{n\in [1,N]}\|u^n-V^n\|_{2,h} \leq c(\tau^2+h^2),
	\end{align}
where the constant $c$ is independent of $h$ and $\tau$.
\end{lm}
\begin{proof}
	Let $r^n\in X_{JK}$ be the consistency error of the method (\ref{soltuonV_globallylipschtiz}) or (\ref{CNFD}), i.e, with $\displaystyle u^{n+1/2}=\frac{u^{n+1}+u^n}{2}$, we define
	\begin{align}
		r^n_{j,k} = i\delta_t^+u^n_{j,k} + \delta^2 u_{j,k}^{n+1/2} + \lambda\tilde{\psi}(u^{n+1}_{j,k},u^n_{j,k}) +i\varepsilon\tilde{\varphi}(u^{n+1}_{j,k},u^n_{j,k}),~~ (j,k)\in\T_{JK}.\label{df:rn}
	\end{align}
	Let $e^n:=u^n-V^n\in X_{JK}$ for all $n = 0,1,\cdots,N$. Subtracting the first equality in \eqref{soltuonV_globallylipschtiz} from the equality \eqref{df:rn}, it implies
	\begin{align}\label{truncationerrorCNFD}
		i\delta_t^+e^n_{j,k} + \delta^2 e_{j,k}^{n+\frac{1}{2}} + \lambda\left(\tilde{\psi}(u^{n+1}_{j,k},u^n_{j,k}) -\tilde{\psi}(V^{n+1}_{j,k},V^n_{j,k}) \right)\nonumber\\ +i\varepsilon \left(\tilde{\varphi}(u^{n+1}_{j,k},u^n_{j,k}) -\tilde{\varphi}(V^{n+1}_{j,k},V^n_{j,k})\right) - r_{j,k}^n  = 0.
	\end{align}
We take the inner product with $e^{n+1/2}$ and take imaginary parts. By applying the Cauchy–Schwarz inequality and using the globally Lipschitz of
$\tilde\varphi$ and $\tilde\psi$, 
it yields  
        \begin{align*}
  \|e^{n+1}\|^2_{2,h} - \|e^{n}\|^2_{2,h} \leq C\tau(\|e^{n+1}\|_{2,h}  + \|e^{n}\|_{2,h} + \|r^n\|_{2,h})\|e^{n+1/2}\|_{2,h}.
	\end{align*}
 From the definition of $r^n_{j,k}$
 and Lemma \ref{lm:est_rn}, one has
	\begin{align*}
		(1-c\tau)\|e^{n+1}\|_{2,h}\leq (1+c\tau)\|e^{n}\|_{2,h} + C\tau(\tau^2+h^2).
	\end{align*}
	The result follows in the view of Gr\"onwall's discrete theorem. 
\end{proof}
{\it Proof of Theorem \ref{thm:err_uandU_L2H01norm}}. We use the obvious inequality
	\begin{align*}
		\|\omega\|_{\infty,h} \leq h\|\omega\|_{2,h},
	\end{align*}
	for $\omega\in X_{JK}$. From (\ref{ineq:errorVandu}), it yields
	\begin{align*}
		\max_{1\leq n\leq N}\max_{(j,k)\in \T_{JK}} |u^n_{j,k}- V_{j,k}^n| \leq c(\frac{\tau^2}{h}+h),
	\end{align*}
	i.e, for $\tau\lesssim h$ and $h$ sufficiently small (bounded by a constant of $\frac{M}{2c}$), it implies $V_{j,k}^n\leq M$, $n=1,\cdots,N$, $(j,k)\in \T_{JK}$. Therefore, $V^n$ satisfies (\ref{CNFD}), i.e, for $h$ small enough and from the unique solution in Theorem \ref{thm:existsol_uniqueSolution_CNFD}, we have $V^n=U^n$. From this result and from (\ref{ineq:errorVandu}), the proof of \eqref{ineq:errorUandu_L2} is completed.\\
	Now, we will estimate the error on norm $|\cdot|_{1,h}$. One can write  (\ref{truncationerrorCNFD}) as  follows:
	\begin{align}\label{truncationerrorCNFD1}
		i\delta_t^+e^n_{j,k} +  \chi_{j,k}^{n+1/2} + \eta_{j,k}^n  + \xi_{j,k}^n - r_{j,k}^n  = 0,
	\end{align}
	where $\chi^n,\xi^n,\eta^n\in X_{JK}$ and 
	\begin{align*}
	    \chi^n_{j,k} &= \delta^2 e_{j,k}^{n+1/2},\\
		\eta_{j,k}^n &=\lambda\left({\psi}(u^{n+1}_{j,k},u^n_{j,k}) - {\psi}(U^{n+1}_{j,k},U^n_{j,k}) \right),\\
		\xi_{j,k}^n & = i\varepsilon\left({\varphi}(u^{n+1}_{j,k},u^n_{j,k}) -{\varphi}(U^{n+1}_{j,k},U^n_{j,k})\right).
	\end{align*}
	From the definition of $\psi$ and $\varphi$, there holds
	\begin{align*}
		\eta_{j,k}^n &= \lambda u^{n+1/2}_{j,k}\int_0^1f\left(|u^n_{j,k}|^2 + t(|u^{n+1}_{j,k}|^2-|u^n_{j,k}|^2)\right)dt\\
        &-\lambda U^{n+1/2}_{j,k}\int_0^1f\left(|U^n_{j,k}|^2 + t(|U^{n+1}_{j,k}|^2-|U^n_{j,k}|^2)\right)dt\\
		&= \lambda U^{n+1/2}_{j,k}\int_0^1\left[f\left(|u^n_{j,k}|^2 + t(|u^{n+1}_{j,k}|^2-|u^n_{j,k}|^2)\right) - f\left(|U^n_{j,k}|^2 + t(|U^{n+1}_{j,k}|^2-|U^n_{j,k}|^2)\right)\right]dt\\ &+\lambda e^{n+1/2}_{j,k}\int_0^1f\left(|u^n_{j,k}|^2 + t(|u^{n+1}_{j,k}|^2-|u^n_{j,k}|^2)\right)dt. 
	\end{align*}
Thus,
	\begin{eqnarray*}
		\eta_{j,k}^n&& = \lambda U^{n+1/2}_{j,k}\int_0^1\left[H^n_{j,k}+t(H^{n+1}_{j,k}-H^n_{j,k})\right]\int_0^1f'(g_1(t,s))dsdt
\nonumber \\&&	 +\lambda e^{n+1/2}_{j,k}\int_0^1f(|u^n_{j,k}|^2 + t(|u^{n+1}_{j,k}|^2-|u^n|^2))dt,
	\end{eqnarray*}
	where $H^n_{j,k}=|u^{n}_{j,k}|^2-|U^n_{j,k}|^2 = e^n_{j,k}\overline{u^n_{j,k}}+U^n_{j,k}\overline{e_j^n}$ and 
 $$
    g_1(t,s) =|U^n_{j,k}|^2 + t(|U^{n+1}_{j,k}|^2-|U^n_{j,k}|^2) +s(H^n_{j,k}+t(H^{n+1}_{j,k}-H^n_{j,k})).
 $$
 There also holds
	\begin{align*}
		\xi^n_{j,k} = i\varepsilon\left[e^n_{j,k}\overline{u^{n+1/2}_{j,k}}+U^{n+1/2}_{j,k}\overline{e_j^{n+1/2}}\right]U^{n+1/2}_{j,k} + i\varepsilon |u^{n+1/2}_{j,k}|^2e^{n+1/2}_{j,k}.
	\end{align*}
	Since $U^{n}, U^{n+1}, u^{n}, u^{n+1}, f, f',f'', |U^{n+1}|_{1,h}, |U^{n}|_{1,h}, |u^{n+1}|_{1,h}, |u^{n}|_{1,h}$ are bounded, one can obtain 
	\begin{align}
		\|\eta^n\|_{2,h} \lesssim \|e^{n+1}\|_{2,h} + \|e^{n}\|_{2,h},\;\;
		\|\delta^+\eta^n\|_{2,h} \lesssim |e^{n+1}|_{1,h} + |e^{n}|_{1,h}\label{estL2H01:eta}
	\end{align}
and
	\begin{align}
			\|\xi^n\|_{2,h} \lesssim \|e^{n+1}\|_{2,h} + \|e^{n}\|_{2,h},\;\;
		\|\delta^+\xi^n\|_{2,h} \lesssim |e^{n+1}|_{1,h} + |e^{n}|_{1,h}.\label{estL2H01:xi}
	\end{align}
 Taking in (\ref{truncationerrorCNFD1}) the inner product with $e^{n+1}-e^{n}$ and taking the real part, there holds
	\begin{align*}
		\E_h(e^{n+1})-\E_h(e^n) &= 2\Re\langle\eta^n+\xi^n - r^n,e^{n+1}-e^n\rangle_h\\
		&=2\Re\langle \eta^n+\xi^n - r^n, i\tau(\chi^n+\eta^n+\xi^n - r^n) \rangle_h\\
		&= -2\tau\Im\langle \eta^n + \xi^n - r^n, \chi^n\rangle_h.
	\end{align*}
	By using the definition of $\chi^n_{j,k}$, it yields
	\begin{align*}
		|\langle \eta^n+\xi^n, \chi^n\rangle_h| &= |\langle \eta^n+\xi^n,\delta^2 e^{n+1/2}\rangle_h|
		= |(\delta^+\eta^n+\delta^+\xi^n ,\delta^+e^{n+1/2})_h|.
	\end{align*}
 Thus,
 	\begin{align*}
		|\langle \eta^n+\xi^n, \chi^n\rangle_h|  \leq  \frac{1}{2}|\eta^n|_{1,h}(|e^{n+1}|_{1,h}+|e^{n}|_{1,h})+ \frac{1}{2}|\xi^n|_{1,h}(|e^{n+1}|_{1,h}+|e^{n}|_{1,h}).
	\end{align*}
	Using the estimation of $|\eta^n|_{1,h}$ and $|\xi^n|_{1,h}$ in \eqref{estL2H01:eta}-\eqref{estL2H01:xi}, the previous inequality implies 
	\begin{align}
		|\langle \eta^n+\xi^n, \chi^n\rangle_h| \lesssim |e^{n}|^2_{1,h} + |e^{n+1}|^2_{1,h}. \label{est:etaxichi}
	\end{align}
By using the similar calculations, one then obtains
	\begin{align*}
		|\langle r^n,\chi^n\rangle_h|& = |\langle r^n,\delta^2 e^{n+1/2}\rangle_h| = |\langle \delta^+ r^n, \delta^+ e^{n+1/2}  \rangle_h| \leq  \frac{1}{2}|r^n|_{1,h}(|e^{n}|_{1,h} + |e^{n+1}|_{1,h}).
	\end{align*}
	From the definition of $r^n$, Lemma \ref{lm:est_rn}, and Lemma \ref{lm:est_H01rn}, there holds
	\begin{align}\label{est:rchi}
		|\langle r^n,\chi^n\rangle_h| \lesssim (\tau^2+h^2)^2 + |e^{n+1}|^2_{1,h} +  |e^{n}|^2_{1,h}.
	\end{align}
	From (\ref{est:etaxichi}) and (\ref{est:rchi}), it arrives at
	\begin{align*}
		\E_h(e^{n+1})-\E_h(e^n) \lesssim \tau(\tau^2+h^2)^2 + \tau(|e^{n+1}|^2_{1,h} +  |e^{n}|^2_{1,h} ).
	\end{align*}
	Therefore,
	\begin{align*}
		\E_h(e^{n+1})-\E_h(e^n) \lesssim \tau(\tau^2+h^2)^2 + \tau(\E_h(e^{n+1})+\E_h(e^n)).
	\end{align*}
	Using Gr\"onwall's inequality, it holds
	\begin{align*}
		\E_h(e^n) \lesssim (\tau^2+h^2)^2.
	\end{align*}
	From the definition of $\E_h(e^n)$ in \eqref{df:E_h1}, one obtains \eqref{ineq:errorUandu_H1}.
\qed

\section{Numerical experiments}
\label{Num}
In this section, we validate the numerical CNFD scheme of (\ref{CNFD}) by the numerical experiments of (\ref{eq:NLS}). 
We consider the initial soliton condition of (\ref{eq:NLS}) in the form \cite{NH2021,Yang2010}
\begin{align}
	u_0(x,y) =A_0 v_0(X_0,Y_0)\exp\left[i\alpha_0 + i\chi_0({X_0},{Y_0})\right], 
	\label{IC}
\end{align}
where 
$A_0$ is the initial amplitude parameter,
$(x_0,y_0)$ is the initial position,
${\bf d} = (d_1,d_2)$ is the velocity vector of the soliton with the velocity components in the $x$ and $y$ directions as $d_1$ and $d_2$, respectively, $X_0 = x-x_0$, $Y_0=y-y_0$,
$\chi_0({X_0},{Y_0}) = \tilde{{d}_{1}}X_0+\tilde{{d}_2}Y_0$ with
$\tilde{{d}_{1}} = d_1/2$ and $\tilde{{d}_{2}} = d_2/2$,
$\alpha_0$ is the initial phase, and
$v_0(X_0,Y_0)$ is the localized real-valued amplitude function.
Letting $A_0=1$ for simplicity, it can be shown that $v_0(X_0,Y_0)$ satisfies the following elliptic equation \cite{NH2021,Yang2010}:
\begin{align}
	\Delta v_0 + \frac{v_0^3}{1+v_0^2} = \mu v_0,
	\label{elliptic_eq1}
\end{align}
where $\mu$ is the propagation constant.
We note that the soliton solution of (\ref{eq:NLS}) for $t>0$ was theoretically explored in Ref. \cite{NH2021}. More specifically, an approximate soliton solution of  (\ref{eq:NLS}) can be found in the form of
\begin{align}
	u(x,y,t) =A(t)v(X,Y)\exp\left[i\alpha + i\tilde\chi(\tilde X,\tilde Y)\right], 
\label{single_soliton1}
\end{align}
where $t>0$,
 $X = x-x_0 - d_1 t$,
 $Y = y-y_0 - d_2 t$,
 $\tilde X = x - x_0 - \tilde{d}_{1}t$,
  $\tilde Y = y - y_0 - \tilde{d}_{2}t$,
$\tilde\chi({\tilde X},{\tilde Y}) = \tilde{{d}_{1}}\tilde X  + \tilde{d}_{2}\tilde Y $, $\alpha$ is related to the phase, 
$v(X,Y)$ is the solution of the following elliptic equation:
\begin{align}
	\Delta v + \frac{v^3}{1+v^2} = \mu v,
	\label{elliptic_eq2}
\end{align}
and the amplitude parameter $A(t)$ satisfies \cite{NH2021}:
\begin{align}
	A(t) = A_0\left[1 + 2\epsilon \|u_0\|^{4}_{4}\|u_0\|^{-2}_{2} A_0^4 t \right]^{-1/2}.
\label{amplitude}
\end{align}
In practice, one can numerically solve 
equations (\ref{elliptic_eq1}) and (\ref{elliptic_eq2}) by a Fourier iteration method such as the accelerated imaginary-time evolution method (AITEM) to define the ground states $v_{0}$ and $v$ \cite{NH2021,Yang2010}, respectively.

In the current work, we compare the soliton profiles and their amplitudes obtained by the CNFD scheme of (\ref{CNFD}) to the ones obtained by the theoretical calculations from equations (\ref{single_soliton1}) and (\ref{amplitude}), respectively, and to the ones obtained by a pseudo-spectral method for numerically solving the propagating soliton solution of (\ref{eq:NLS}) such as the split-step Fourier method (SSFM) with the second-order accuracy \cite{Yang2010}. We also present the order convergence of the CNFD scheme. For this purpose, 
we first define the relative error in measuring the amplitude parameter $A(t)$ at the time $t$ as follows:
\begin{itemize}
\item $E_A^{(CNFD)} = |A^{(CNFD)} - A^{(th)}|/A^{(th)}$,
\item $E_A^{(SSFM)} = |A^{(SSFM)} - A^{(th)}|/A^{(th)}$,
\end{itemize}
where $A^{(CNFD)}$ and $A^{(SSFM)}$ are measured by the simulations of  (\ref{eq:NLS}) using the CNFD scheme and the SSFM scheme, respectively, while 
$A^{(th)}$ is measured from  (\ref{amplitude}).
Additionally, we also define the relative difference in measuring the amplitude parameter $A(t)$ at the time $t$ obtained by the CNFD scheme to the one obtained by the SSFM scheme: 
$$D_{A} = | A^{(CNFD)}-A^{(SSFM)}|/A^{(SSFM)}.$$
Second, we define the relative error in measuring the soliton solution profile of $u(x,y,t)$ at the time $t$ as follows:
\begin{itemize}
\item $E_{2,h}^{(CNFD)} = \big\lVert|u^{(CNFD)}| - |u^{(th)}|\big\rVert_{2,h}/\big\lVert u^{(th)}\big\rVert_{2,h}$,
\item $E_{2,h}^{(SSFM)} = \big\lVert|u^{(SSFM)}| - |u^{(th)}|\big\rVert_{2,h}/\big\lVert u^{(th)}\big\rVert_{2,h}$,
\item $E_{1,h}^{(CNFD)} = \big| |u^{(CNFD)}| - |u^{(th)}|\big|_{1,h}/\big| |u^{(th)}|\big|_{1,h}$,
\item $E_{1,h}^{(SSFM)} = \big| |u^{(SSFM)}| - |u^{(th)}|\big|_{1,h}/\big| |u^{(th)}|\big|_{1,h}$,
\end{itemize}
where $u^{(CNFD)}$ and $u^{(SSFM)}$ are measured by the simulations of  (\ref{eq:NLS}) using the CNFD scheme and the SSFM scheme, respectively, while 
$u^{(th)}$ is measured from  (\ref{single_soliton1}).
Additionally, we also define the relative difference in measuring the solution $u(x,y,t)$ at the time $t$ obtained by the CNFD scheme to the one obtained by the SSFM scheme as follows:
\begin{itemize}
\item $D_{2,h} = \big\lVert|u^{(CNFD)}| - |u^{(SSFM)}|\big\rVert_{2,h}/\big\lVert u^{(SSFM)}\big\rVert_{2,h}$,
\item $D_{1,h} = \big| |u^{(CNFD)}| - |u^{(SSFM)}|\big|_{1,h}/\big| |u^{(SSFM)}| \big|_{1,h}$.
\end{itemize}

As a concrete example, the simulations are carried out on the domain $\Omega = [-40,40]\times [-40,40]$ with $\lambda =1$ and $\varepsilon = 0.01$. 
The parameters for the initial condition are: $A_0=1$, $(x_0,y_0)=(-5, 4.5)$, $(d_1,d_2)=(2, -1.8)$, and $\alpha_0=0$. 
Moreover, the ground state $v_0$ in (\ref{IC}) is computed by the simulation of (\ref{elliptic_eq1}) using the AITEM scheme. In the simulations of (\ref{elliptic_eq1}),  we use the input function $\tilde v_0 = \sech\left[(x-x_0)^2 + (y-y_0)^2\right]$
and use the input power value of the beam as $\tilde P_0 =22.5$. One then obtains the power value of the soliton $u_0$ as $P=22.5$ at $\mu=0.1692$. We also use the final time $t_f=5$, as an example, in the simulations of (\ref{eq:NLS}) with the CNFD scheme and with the SSFM scheme.  
Furthermore, to solve nonlinear system \eqref{CNFD} using the CNFD scheme, the system is linearized by the fixed point method:
\begin{align*}
	\left\{
	\begin{array}{c}
		i\frac{U^{{n,l+1}}-U^n}{\tau}+\bigg[\delta^2+\lambda\frac{F(|U^{n,l}|^2)-F(|U^{n}|^2)}{|U^{n,l}|^2-|U^{n}|^2} + i\varepsilon |\frac{U^{n,l}+U^{n}}{2}|^2\bigg]\frac{U^{l+1,n}+U^{n}}{2} =0 ~~\forall (j,k) \in \T_{JK}, l\geq 0,\\
		U^{n,0}= U^n;~U^{n,l}_{j,k}=U^n_{j,k}= 0, ~\forall (j,k) \in \T^0_{JK};~U^{n,l+1}\in X_{JK}.
	\end{array}
	\right.
\end{align*}
 It can be shown that $\{U^{n,l}\}_{l=0,1,\cdots}$ converges $U^{n+1}$. In practice, one can choose $U^{n+1} = U^{n,l+1}$ when
\begin{align*}
	\frac{\|U^{n,l+1}-U^{n,l}\|_{2,h}}{\|U^{n,l+1}\|_{2,h}} \leq 10^{-8}.
\end{align*}  

We first present the simulation results for the mesh size $h=2^{-4}$ and time step $\tau=2^{-7}$.  
Figure \ref{fig1} shows the initial soliton profile $|u(x,y,0)|$ and its evolution profiles $|u(x,y,t)|$ obtained by the simulation of (\ref{eq:NLS}) using the CNFD scheme of (\ref{CNFD}). At each time $t=0,1,2,3,4,5$, the soliton profile is shown by using the level colormap and its soliton amplitude value is calculated.  
We observe that $E_A^{(CNFD)}\le 3.2745$E-4, $E_A^{(SSFM)}\le 3.4993$E-4, and $D_{A}\le 2.2474$E-5 for $0<t\le t_f=5$.

Next, we show the simulation results  with varying values of $h$ and $\tau$.
Table \ref{table3} presents the relative errors in measuring the soliton profiles obtained by the simulations and the theoretical calculations of  (\ref{single_soliton1}) while Table \ref{table4} shows a comparison for the soliton profiles obtained by the CNFD scheme and the SSFM scheme. Furthermore, on the lines 2 and 4 of each row of Table \ref{table4}, the second order convergence rate of the CNFD scheme for $\|\cdot\|_{2,h}$-norm and $|\cdot|_{1,h}$-norm are illustrated, respectively. 
These numerical results above validated our CNFD scheme.

\begin{figure}[h]
\begin{center}
\centering
\includegraphics[width=1.0\linewidth]{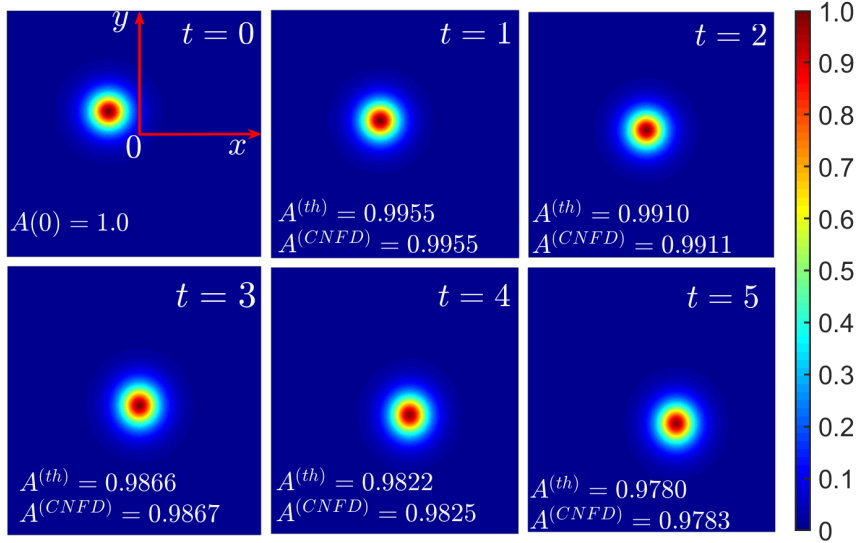} 
\caption{(Color online) The soliton profile at $t=0$ and its evolution obtained by the simulation of (\ref{eq:NLS}) using the CNFD scheme of  (\ref{CNFD}) with parameters as $h=2^{-4}$ and $\tau=2^{-7}$.
}
\label{fig1}
\end{center}
\end{figure}

\begin{table}{ }
	\centering
	\begin{tabular}{ |c|c|c|c|c|c|} 
		\hline Time
		 & Relative & $h=2^{-2}$ &  $h=2^{-3}$ &  $h=2^{-4}$ &  $h=2^{-5}$ \\  
		 & error  & $\tau=2^{-5}$ &  $\tau=2^{-6}$ &  $\tau=2^{-7}$ &  $\tau=2^{-8}$ \\ 
		\hline
		\multirow{4}{*}{t=1}
		& $E_{2,h}^{(CNFD)}$ &8.7040E-3  &3.3323E-3  &2.8474E-3  &2.8622E-3 \\ 
		& $E_{2,h}^{(SSFM)}$ &2.8735E-3  &2.8767E-3  &2.8784E-3   &2.8792E-3  \\ 
		& $E_{1,h}^{(CNFD)}$ &1.6161E-2  &7.0055E-3  &6.6400E-3   &6.7738E-3 \\ 
		& $E_{1,h}^{(SSFM)}$ &6.8024E-3  &6.8226E-3  &6.8308E-3   &6.8344E-3  \\ 
		\hline
		\multirow{4}{*}{t=3}
		& $E_{2,h}^{(CNFD)}$ &2.5497E-2  &8.5926E-3  &7.5126E-3  &7.7461E-3 \\ 
		& $E_{2,h}^{(SSFM)}$ &7.8298E-3  &7.8479E-3  &7.8574E-3   &7.8622E-3  \\ 
		& $E_{1,h}^{(CNFD)}$ &4.5495E-2  &1.4835E-2  &1.4236E-2   &1.5012E-2 \\ 
		& $E_{1,h}^{(SSFM)}$ &1.5222E-2  &1.5283E-2  &1.5312E-2   &1.5327E-2  \\ 
		\hline
		\multirow{4}{*}{t=5}
		& $E_{2,h}^{(CNFD)}$ &4.2145E-2  &1.3519E-2  &1.2217E-2  &1.2779E-2 \\ 
		& $E_{2,h}^{(SSFM)}$ &1.2963E-2  &1.2999E-2  &1.3018E-2   &1.3028E-2  \\ 
		& $E_{1,h}^{(CNFD)}$ &7.2319E-2  &2.1808E-2  &2.1222E-2   &2.2646E-2 \\ 
		& $E_{1,h}^{(SSFM)}$ &2.3040E-2  &2.3137E-2  &2.3184E-2   &2.3208E-2  \\ 
		\hline
	\end{tabular}
	\caption{The relative errors in measuring the soliton profiles obtained by the numerical experiments using the CNFD scheme and SSFM scheme and the theoretical calculations of  (\ref{single_soliton1}).
	}	\label{table3}
\end{table}
\begin{table}[t]
	\centering
	\begin{tabular}{ |c|c|c|c|c|c|} 
		\hline Time
		 & Relative difference & $h=2^{-2}$ &  $h=2^{-3}$ &  $h=2^{-4}$ &  $h=2^{-5}$ \\  
		 & and convergence rate  & $\tau=2^{-5}$ &  $\tau=2^{-6}$ &  $\tau=2^{-7}$ &  $\tau=2^{-8}$ \\ 
		\hline
		\multirow{4}{*}{t=1}
		&$D_{2,h}$&8.6271E-3  &2.1669E-3  &5.4218E-4  &1.3550E-4 \\ 
		& Rate    &1.9907  &1.9983  &2.0007  & \\ 
		&$D_{1,h}$&1.6437E-2  &4.1285E-3  &1.0324E-3  &2.5805E-4 \\ 
		& Rate    &1.9906  &1.9994  &2.0004  & \\ 
		\hline
		\multirow{4}{*}{t=2}
		&$D_{2,h}$&1.7503E-2  &4.3863E-3  &1.0963E-3  &2.7365E-4 \\ 
		& Rate    &1.9952  &2.0004  &2.0031  & \\ 
		&$D_{1,h}$&3.3373E-2  &8.3157E-3  &2.0731E-3  &5.1620E-4 \\ 
		& Rate    &2.0066  &2.0056  &2.0081  & \\ 
		\hline
		\multirow{4}{*}{t=3}
		&$D_{2,h}$&2.6596E-2  &6.6441E-3  &1.6588E-3  &4.1364E-4 \\ 
		& Rate    &2.0015  &2.0027  &2.0051  & \\ 
		&$D_{1,h}$&4.9961E-2  &1.2350E-2  &3.0715E-3  &7.6360E-4 \\ 
		& Rate    &2.0227  &2.0104  &2.0112  & \\ 
		\hline
		\multirow{4}{*}{t=4}
		&$D_{2,h}$&3.5779E-2  &8.9094E-3  &2.2220E-3  &5.5373E-4 \\ 
		& Rate    &2.0080  &2.0048  &2.0064  & \\ 
		&$D_{1,h}$&6.5798E-2  &1.6164E-2  &4.0135E-3  &9.9715E-4 \\ 
		& Rate    &2.0353  &2.0137  &2.0125  & \\ 
		\hline
		\multirow{4}{*}{t=5}
		&$D_{2,h}$&4.4951E-2  &1.1160E-2  &2.7810E-3  &6.9270E-4 \\ 
		& Rate    &2.0139  &2.0066  &2.0074  & \\ 
		&$D_{1,h}$&8.0890E-2  &1.9780E-2  &4.9057E-3  &1.2185E-3 \\ 
		& Rate    &2.0449  &2.0160  &2.0130  & \\ 
		\hline
	\end{tabular}
	\caption{The relative differences in measuring the soliton profiles obtained by the CNFD scheme and SSFM scheme and the convergence rate of the CNFD scheme. }\label{table4}
\end{table}

		
		

\section{Conclusion}
\label{Con}

The (2+1)D saturable NLS equation, which is of a common class of (2+1)D NLS equations with modified nonlinearity, can stabilized the 2D solitons at high speed and high power. Due to the nonintegrability of the (2+1)D saturable NLS equation, its stability, and its extensive applications in guided-wave optics and 2D materials, the numerical techniques for the (2+1)D saturable NLS have recently become one of the central issues in 2D material simulations. In this work, we established and analyzed the CNFD scheme for the (2+1)D saturable nonlinear Schr\"odinger equation in a 2D nonlinear optical medium with cubic loss. This perturbation can arise due to the presence of presence of TPA in optical materials such as silicon photonic crystals.  By appropriate settings and estimations, we proved the existence and the uniqueness of the numerical solution of the (2+1)D saturable NLS model. Under a mild condition of $\tau\lesssim h$, we showed that the convergence rate is at order $O(\tau^2+h^2)$ in $l^2-$norm and discrete $H^1-$norm. To validate the numerical scheme, we intensively simulated the (2+1)D NLS model by the CNFD scheme with varying both spatial mesh size $h$ and time step $\tau$. Additionally, we implemented the CNFD to validate an theoretical expression for the amplitude dynamics of a single 2D soliton propagation under the effect of cubic loss. Moreover, we also verified the numerical results of the CNFD scheme with the conventional SSFM scheme. We obtained an excellent agreement between the numerical simulation results and the theoretical results. More specifically, for $h=2^{-3}$ and $\tau=2^{-6}$, at $t=1$, the relative errors in measuring the soliton profiles obtained by the numerical experiments using the CNFD scheme and SSFM scheme and the theoretical calculations of equation (\ref{single_soliton1}) are in the same order of $10^{-3}$ and at $t=5$, these relative errors are in the same order of $10^{-2}$. The CNFD scheme proposed in this work can open a way to numerically study the propagation of 3D soliton (also known as light bullets) in a modified nonlinear medium under framework of a saturable $(3+1)$D NLS equation with a deterministic or a random noise.

\appendix
\section{Estimate the consistency error}
\label{Appendix}
\begin{lm}\label{lm:est_rn}
	Let the solution of (\ref{eq:NLS}) be smooth enough and let $r^n\in X_{JK}$ defined by \eqref{df:rn} be the consistency error of the method (\ref{soltuonV_globallylipschtiz}) or (\ref{CNFD}). Then
	\begin{align*}
		\max_{n}|r^n_{j,k}| \leq c(\tau^2+h^2),~~~~\forall  (j,k)\in\T_{JK}.
	\end{align*}
\end{lm}
\begin{proof}
	From definition of $r_{j,k}^n$, there holds
	\begin{align*}
		\!r^n_{j,k} = i\delta^+_t u^{n}_{j,k}
        +\frac{1}{2}(\delta^2_{x}u^{n+1}_{j,k} + \delta^2_{x}u^{n}_{j,k}) + \frac{1}{2}(\delta^2_{y}u^{n+1}_{j,k} + \delta^2_{y}u^{n}_{j,k}) +\lambda\psi(u^{n+1}_{j,k},u^{n}_{j,k}) +i\varepsilon \varphi(u^{n+1}_{j,k},u^{n}_{j,k}),
	\end{align*}
	where we remind that $u^{n}_{j,k}=u(x_j,y_k,t_{n})$ for all $j=\overline{0,J}$, $k=\overline{0,K}$ and $n=\overline{0,N}$.\\ Now we will estimate the terms of right hand side of previous definition. Applying the Taylor series, there exists $s_1\in(t_{n},t_{n+1})$ such that
	\begin{align}
		\bigg|\delta^+_t u^{n}_{j,k} - u_t(x_j,y_k,t_{n+1/2})\bigg| \leq \frac{|u_{ttt}(x_j,y_k,s_1)|}{8}\tau^2, \label{ineq:est_dt}
	\end{align}
 where $t_{n+1/2}$ is the midpoint of $(t_{n},t_{n+1})$.
 Moreover, there exist $\alpha_1,\alpha_2\in(x_{j-1},x_{j+1})$, $\beta_1,\beta_2\in (y_{k-1},y_{k+1}),$ and $s_2,s_3\in(t_{n},t_{n+1})$ such that
	\begin{align*}
		&\big|\delta^2_{x}u^{n}_{j,k} - u_{xx}(x_j,y_k,t_{n})\big| \leq \frac{|u_{xxxx}(\alpha_1,y_k,t_{n})|}{12}h^2,\\
        &\big|\delta^2_{x}u^{n+1}_{j,k} - u_{xx}(x_j,y_k,t_{n+1})\big| 
        \leq \frac{|u_{xxxx}(\alpha_2,y_k,t_{n+1})|}{12}h^2,\\
		&\big|u_{xx}(x_j,y_k,t_{n}) + u_{xx}(x_j,y_k,t_{n+1})-2u_{xx}(x_j,y_k,t_{n+1/2})\big| \leq \frac{|u_{xxtt}(x_j,y_k,s_2)|}{4}\tau^2.
	\end{align*}
	From three previous inequalities, one can obtain
	\begin{align}
		&\big|\delta^2_{x}u^{n}_{j,k}+\delta^2_{x}u^{n+1}_{j,k} - 2u_{xx}(x_j,y_k,t_{n+1/2})\big| \nonumber\\
		&\leq \frac{|u_{xxxx}(\alpha_1,y_k,t_{n})|}{12}h^2 + \frac{|u_{xxxx}(\alpha_2,y_k,t_{n+1})|}{12}h^2 +\frac{|u_{xxtt}(x_j,y_k,s_2)|}{4}\tau^2. \label{ineq:est_dxx}
	\end{align}
    Similarly, we have
	\begin{align*}
		&\big| \delta^2_{y}u^{n}_{j,k}- u_{yy}(x_j,y_k,t_{n})\big| \leq \frac{|u_{yyyy}(x_j,\beta_1,t_{n})|}{12}h^2,\\
		&\big|\delta^2_{y}u^{n+1}_{j,k} - u_{yy}(x_j,y_k,t_{n+1})\big| \leq \frac{|u_{yyyy}(x_j,\beta_2,t_{n+1})|}{12}h^2,\\
		&\big|u_{yy}(x_j,y_k,t_{n}) + u_{yy}(x_j,y_k,t_{n+1})-2u_{yy}(x_j,y_k,t_{n+1/2})\big| \leq \frac{|u_{yytt}(x_j,y_k,s_3)|}{4}\tau^2.
	\end{align*}
	From three previous inequalities, it yields
	\begin{align}
		&\big|\delta^2_{y}u^{n}_{j,k} +\delta^2_{y}u^{n+1}_{j,k} -2u_{yy}(x_j,y_k,t_{n+1/2})\bigg|\nonumber\\
		&\leq \frac{|u_{yyyy}(x_j,\beta_1,t_{n})|}{12}h^2 +\frac{|u_{yyyy}(x_j,\beta_2,t_{n+1})|}{12}h^2 + \frac{|u_{yytt}(x_j,y_k,s_3)|}{4}\tau^2. \label{ineq:est_dyy}
	\end{align}
	There exists $s_4\in (t_{n},t_{n+1})$ such that
	\begin{align}
		\bigg|\bigg|\frac{u(x_{j},y_{k},t_{n}) + u(x_{j},y_{k},t_{n+1})}{2}\bigg|^2 - |u(x_{j},y_{k},t_{n+1/2})|^2\bigg| &\leq \frac{|u_{tt}(x_j,y_k,s_4)|^2}{16}\tau^4,\label{ineq:est_binhphuong}\\
		\bigg|\frac{u(x_{j},y_{k},t_{n}) + u(x_{j},y_{k},t_{n+1})}{2} - u(x_{j},y_{k},t_{n+1/2})\bigg| &\leq \frac{|u_{tt}(x_j,y_k,s_4)|}{4}\tau^2. \label{ineq:est_trungbinh}
	\end{align}
	For $z_1,z_2\in \R$, one can get
	\begin{align*}
		&L_2=\int^1_0f(z_1+t(z_2-z_1))dt-f(\frac{z_2+z_1}{2})\\
		&=\int_0^{1/2}f(z_1+t(z_2-z_1))dt + \int_0^{1/2}f(z_1+(1-t)(z_2-z_1))dt-f(\frac{z_2+z_1}{2})\\
		&=\int_0^{1/2}[f(z_1+t(z_2-z_1))-f(\frac{z_2+z_1}{2})]dt + \int_0^{1/2}[f(z_1+(1-t)(z_2-z_1))-f(\frac{z_2+z_1}{2})]dt.
	\end{align*}
It implies,
	\begin{align*}
		L_2&=\int_0^{1/2}(t-\frac{1}{2})(z_2-z_1)\int_0^1f'\bigg(\frac{z_2+z_1}{2}+s(t-\frac{1}{2})(z_2-z_1)\bigg)dsdt\\
		&-\int_0^{1/2}(t-\frac{1}{2})(z_2-z_1)\int_0^1f'\bigg(\frac{z_2+z_1}{2}-s(t-\frac{1}{2})(z_2-z_1)\bigg)dsdt\\
		&=\int_0^{1/2}(t-\frac{1}{2})(z_2-z_1)\int_0^1\bigg[f'\bigg(\frac{z_2+z_1}{2}+s(t-\frac{1}{2})(z_2-z_1)\bigg)\\
        &-f'\bigg(\frac{z_2+z_1}{2}-s(t-\frac{1}{2})(z_2-z_1)\bigg)\bigg]dsdt.
	\end{align*}
Therefore,
	\begin{align*}
		L_2
		=2\int_0^{1/2}(t-\frac{1}{2})^2(z_2-z_1)^2\int_0^1s\int_0^1 f''(g_2(t,s,s_1))ds_1dsdt,
	\end{align*}
    where
    $$
        g_2(t,s,s_1) = \frac{z_2+z_1}{2}-s(t-\frac{1}{2})(z_2-z_1)+2s_1s(t-\frac{1}{2}(z_2-z_1)).
    $$
	Since $|f''|$ is bounded by $1$, previous expression and definition of $L_2$ yield  
	\begin{align*}
		\bigg|\int^1_0f(z_1+t(z_2-z_1))dt-f\bigg(\frac{z_2+z_1}{2}\bigg)\bigg| \leq \frac{|z_2-z_1|^2}{24}.
	\end{align*}
	By choosing $z_1=|u(x_j,y_k,t_{n})|^2$ and $z_2=|u(x_j,y_k,t_{n+1})|^2$, the previous inequality can be written as
	\begin{align*}
		&\bigg|\int^1_0f(|u(x_j,y_k,t_{n})|^2+t(|u(x_j,y_k,t_{n+1})|^2-|u(x_j,y_k,t_{n})|^2))dt\\
  &-f\bigg(\frac{|u(x_j,y_k,t_{n+1})|^2+|u(x_j,y_k,t_{n})|^2}{2}\bigg)\bigg|
		 \leq \frac{||u(x_j,y_k,t_{n+1})|^2-|u(x_j,y_k,t_{n})|^2|^2}{24}.
	\end{align*} 
 We then have the following relations.
	There exists $s_{5}\in(t_n,t_{n+1})$ such that
	\begin{align}
		&\bigg|\int^1_0f(|u(x_j,y_k,t_{n})|^2+t(|u(x_j,y_k,t_{n+1})|^2-|u(x_j,y_k,t_{n})|^2))dt\nonumber\\
 & -f\bigg(\frac{|u(x_j,y_k,t_{n+1})|^2+|u(x_j,y_k,t_{n})|^2}{2}\bigg)\bigg|	\leq \frac{||u(x_j,y_k,s_{5})|^2_t|^2}{24}\tau^2.\label{ineq:est_intf}
	\end{align} 
	There exists $\theta$ between  $\frac{|u(x_j,y_k,t_{n+1})|^2+|u(x_j,y_k,t_{n})|^2}{2}$ and $|u(x_j,y_k,t_{n+1/2})|^2$ such that
	\begin{align*}
	&f\bigg(\frac{|u(x_j,y_k,t_{n+1})|^2+|u(x_j,y_k,t_{n})|^2}{2}\bigg)-f(|u(x_j,y_k,t_{n+1/2})|^2)\\
	&=\bigg[\bigg(\frac{|u(x_j,y_k,t_{n+1})|^2+|u(x_j,y_k,t_{n})|^2}{2}\bigg)-|u(x_j,y_k,t_{n+1/2})|^2\bigg]f'(\theta).
	\end{align*}
	There exists $s_{6}\in (t_n,t_{n+1})$ such that
	\begin{align*}
		\bigg|\bigg(\frac{|u(x_j,y_k,t_{n+1})|^2+|u(x_j,y_k,t_{n})|^2}{2}\bigg)-|u(x_j,y_k,t_{n+1/2})|^2\bigg|\leq \frac{||u(x_j,y_k,s_{6})|^2_{tt}|}{2}\tau^2.
	\end{align*}
	Using $|f'(s)|\leq 1$ for all $s\geq 0$, it arrives at
	 \begin{align}
		\!\!\bigg|f\bigg(\frac{|u(x_j,y_k,t_{n+1})|^2+|u(x_j,y_k,t_{n})|^2}{2}\bigg)-f(|u(x_j,y_k,t_{n+1/2})|^2)\bigg|\leq \frac{||u(x_j,y_k,s_{6})|^2_{tt}|}{2}\tau^2.\label{ineq:est_fbinhphuong}
	\end{align}
	Applying the triangular inequality from two inequalities  \eqref{ineq:est_intf} and \eqref{ineq:est_fbinhphuong}, it implies
	\begin{align}
		\bigg|\int^1_0f(|u(x_j,y_k,t_{n})|^2+t(|u(x_j,y_k,t_{n+1})|^2&-|u(x_j,y_k,t_{n})|^2))dt-f(|u(x_j,y_k,t_{n+1/2})|^2)\bigg|\nonumber\\
		&\leq \frac{12||u(x_j,y_k,s_{6})|^2_{tt}|+||u(x_j,y_k,s_{5})|_t^2|^2}{24}\tau^2.\label{ineq:est_fmidpoint}
	\end{align}
From inequalities \eqref{ineq:est_dt}, \eqref{ineq:est_dxx}, \eqref{ineq:est_dyy}, \eqref{ineq:est_binhphuong}, \eqref{ineq:est_trungbinh}, \eqref{ineq:est_fmidpoint}, the definition of $r^n_{j,k}$, and the first equation \eqref{eq:NLS} at the point $(x_j,y_k,t_{n+1/2})$, there exits $c>0$ such that
	\begin{align*}
		|r^n_{j,k}|\leq c(\tau^2+h^2).
	\end{align*}
The proof is completed.
\end{proof}

Note that one can show the following lemma by a similar argument with some appropriate estimates.
\begin{lm}\label{lm:est_H01rn}
	Let the solution of (\ref{eq:NLS}) be smooth enough and let $r^n\in X_{JK}$ defined by \eqref{df:rn} be the consistency error of the method (\ref{soltuonV_globallylipschtiz}) or (\ref{CNFD}). Then
	\begin{align*}
		\max_{n}|\delta^+r^n_{j,k}| \leq c(\tau^2+h^2),~~~~ ~~~~\forall~j=0,1,\cdots,J-1,~k=0,1,\cdots,K-1.
	\end{align*}
\end{lm}
\begin{proof}
	For $j=0,1,\cdots,J-1,~k=0,1,\cdots,K-1$, one can follow the analogous line for Lemma \ref{lm:est_rn} to obtain $\max_{n}|\delta^+r^n_{j,k}| \leq c(\tau^2+h^2)$. Noting that $\partial_{xx}u=\partial_{yy} u= \partial_{xxxx}u=\partial_{yyyy} u =0$ on the boundary $\partial\Omega$. Therefore, $\max_{n}|\delta^+r^n_{j,k}| \le c(\tau^2+h^2)$ for all $j=0,1,\cdots,J-1,~k=0,1,\cdots,K-1$.
\end{proof}



{}


\begin{thebibliography}{}


\bibitem{Ablowitz2011} Ablowitz, M.J.: Nonlinear Dispersive Waves: Asymptotic Analysis and Solitons. Cambridge University Press, Cambridge (2011)

\bibitem{Agrawal2003} Agrawal, G.P., Kivshar, Y.S.: Optical solitons: From Fibers to Photonic Crystals. Academic Press, San Diego (2003)

\bibitem{Bao2013a} Antoine, X., Bao, W., Besse,C.: Computational methods for the dynamics of the nonlinear Schr\"odinger/Gross–Pitaevskii equations. Comput. Phys. Commun. {\bf 184}(12), 2621-2633 (2013)

\bibitem{Bao2013b} Bao, W., Cai, Y.: Optimal error estimates of finite difference methods for the Gross-Pitaevskii equation with angular momentum rotation. Math. Comp. {\bf 82}, 99-128 (2013)

\bibitem{Bao2003} Bao, W., Jaksch, D., Markowich, P.A.: Numerical solution of the Gross-Pitaevskii equation for Bose–Einstein condensation. 
J. Comput. Phys. {\bf 187}(1), 318-342 (2003)

\bibitem{Husko2014} Blanco-Redondo, A., Husko, C., Eades, D., Zhang, Y., Li, J., Krauss, T.F., Eggleton, B.J.:  Observation of soliton compression in silicon photonic crystals. Nat. Commun. {\bf 5},  3160 (2014)

\bibitem{Torner_2010} Borovkova, O.V., Kartashov, Y.V., Torner, L.:
Stabilization of two-dimensional solitons in cubic-saturable nonlinear lattices.
Phys. Rev. A {\bf 81}, 063806 (2010)

\bibitem{Boyd-2008} Boyd, R.W.: Nonlinear Optics. Academic, San Diego (2008)

\bibitem{Malomed_2004} Chen, Y.F., Beckwitt, K., Wise, F.W., Malomed, B.A.: Criteria for the experimental observation of multidimensional optical solitons in saturable media. Phys. Rev. E {\bf 70}, 046610 (2004)

\bibitem{Nature2003} Christodoulides, D.N., Lederer, F., Silberberg, Y.: Discretizing light behaviour in linear and nonlinear waveguide lattices. Nature {\bf 424}, 817-823 (2003) 

\bibitem{Horton1996} Horton, W., Ichikawa, Y.H.: Chaos and structure in nonlinear plasmas.  World Scientific, Singapore (1996)


\bibitem{Malomed2011} Kartashov, Y.V., Malomed, B.A., Torner, L.: Solitons in nonlinear lattices. Rev. Mod. Phys. {\bf 83}, 247 (2011)


\bibitem{Maia2013} Maia, L.A, Montefusco, E., Pellacci, B.:
Weakly coupled nonlinear Schr\"odinger systems:
the saturation effect.
Calc. Var. Partial Differ. Equ. {\bf 46}, 325-351 (2013)

\bibitem{NH2021} Nguyen, Q.M., Huynh, T.T.: Collision-induced amplitude dynamics of fast 2D solitons in saturable nonlinear media with weak nonlinear loss. Nonlinear Dyn. {\bf 104}, 4339-4353 (2021)

\bibitem{Peleg2019} Peleg, A., Chakraborty, D.:
Radiation dynamics in fast soliton collisions in the presence of cubic loss.  Physica D {\bf 406}, 132397 (2020)

\bibitem{Taha1984} Taha, T.R., Ablowitz, M.I.: Analytical and numerical aspects of certain
nonlinear evolution equations. II. Numerical, nonlinear Schr\"odinger equation.
J. Comput. Phys. {\bf 55}(2),  203-230 (1984)

\bibitem{Malomed2002} Towers, I., Malomed, B.A.: Stable (2+1)-dimensional solitons in a layered medium with sign-alternating Kerr nonlinearity. J. Opt. Soc. Am. B {\bf 19}, 537-543 (2002)

\bibitem{Loon2018} Van Loon, M.A.W, Stavrias, N., Le, N.H.,
Litvinenko, K.L., Greenland, P.T., Pidgeon, C.R., Saeedi, K., 
Redlich, B., Aeppli, G., Murdin, B.N.: Giant multiphoton absorption for THz resonances in silicon hydrogenic donors.
Nat. Photonics {\bf 12}, 179-184 (2018)

\bibitem{Weideman1986} Weideman, J.A.C., Herbst, B.M.: Split-step methods for the solution of the nonlinear Schr\"odinger equation. 
SIAM J. Numer. Anal. {\bf 23}(3), 485-507 (1986)

\bibitem{Weideman1997} Weideman, J.A.C., Herbst, B.M.: Finite difference methods for an AKNS eigenproblem. Math. Comput. Simul. {\bf 43}, 77-88 (1997)

\bibitem{Weilnau2022} Weilnau, C., Ahles, M., Petter, J., Tr\"ager, D., Schr\"oder, J., Denz, C.: Spatial optical (2+1)-dimensional scalar- and vector-solitons in saturable nonlinear media. Ann. Phys. (Leipzig)  {\bf 11}, 573 (2002)

\bibitem{Taha2003} Xu, X., Taha, T.: Parallel Split-Step Fourier Methods for Nonlinear Schr\"odinger-Type Equations. 
J. Math. Model. Algorithms {\bf 2}, 185-201 (2003)




\bibitem{Yang2010} Yang, J.: Nonlinear Waves in Integrable and Nonintegrable Systems. SIAM, Philadelphia (2010)

\bibitem{Yang2008} Yang, J., Lakoba, T.I.: Accelerated imaginary-time evolution methods for the computation of solitary waves. Stud. Appl. Math. {\bf 120}, 265-292 (2008)


\end{thebibliography}
\end{document}